\documentclass{amsart}
\usepackage{amssymb}
\usepackage{latexsym}
\usepackage{amsmath}
\usepackage{euscript}
\usepackage{mathrsfs}

      \def\dC{{\mathbb C}}

      \def\dR{{\mathbb R}}

\def\cA{{\EuScript A}}   \def\cB{{\EuScript B}}   \def\cC{{\EuScript C}}
      
   \def\cH{{\EuScript H}}

   \def\cT{{\EuScript T}}

\def\bm\chi{\mbox{\boldmath$\chi$}}

\def\RE{{\rm \,Re\,}}
\def\IM{{\rm \,Im\,}}

\def\min{{\rm min\,}}

\def\ker{{\rm ker\,}}
\def\ran{{\rm ran\,}}

\def\dom{{\rm dom\,}}

\let\xker=\ker \def\ker{{\xker\,}}

\newcommand{\sgn}{\mathop{\mathrm{sgn}}\nolimits}

\def\cmr{{\dC \setminus \dR}}

\def\senki{{\lbrack\negthinspace [\bot ]\negthinspace\rbrack}}
\def\senki+{{\lbrack\negthinspace [+] \negthinspace\rbrack}}

\newtheorem{theorem}{Theorem}[section]
\newtheorem{proposition}[theorem]{Proposition}

\newtheorem{lemma}[theorem]{Lemma}
\theoremstyle{definition}

\numberwithin{equation}{section}

\usepackage[normalem]{ulem}
\usepackage{color}
\definecolor{DarkBlue}{rgb}{0,0.1,0.7}
\newcommand{\txtD}{\textcolor{DarkBlue}}
\newcommand\soutD{\bgroup\markoverwith
{\textcolor{DarkBlue}{\rule[.01ex]{2pt}{1pt}}}\ULon}
\newcommand{\Hm}[1]{\leavevmode{\marginpar{\tiny%
$\hbox to 0mm{\hspace*{-0.5mm}$\leftarrow$\hss}%
\vcenter{\vrule depth 0.1mm height 0.1mm width \the\marginparwidth}%
\hbox to
0mm{\hss$\rightarrow$\hspace*{-0.5mm}}$\\\relax\raggedright #1}}}

\begin{document}

%\nocite{*}

\title{An indefinite Laplacian on a rectangle}
 \author[J.~Behrndt]{Jussi Behrndt}
 \author[D.~Krej\v{c}i\v{r}\'ik]{David Krej\v{c}i\v{r}\'ik}

%\date{29 July 2014}

\address{Institut f\"ur Numerische Mathematik\\
Technische Universit\"at Graz \\
Steyrergasse 30\\
8010 Graz \\
Austria}
\email{behrndt@tugraz.at}

\address{Department of Theoretical Physics \\
Nuclear Physics Institute ASCR \\
25068 \v{R}e\v{z} \\ Czech Republic}
\email{krejcirik@ujf.cas.cz}

\begin{abstract}
In this note we investigate the nonelliptic differential expression $\cA=-\text{div}\,\sgn\nabla$ on a rectangular domain $\Omega$ in the plane. 
The seemingly simple problem to associate a selfadjoint operator with the differential expression $\cA$ in $L^2(\Omega)$ is solved here. 
Such indefinite Laplacians arise in mathematical models of metamaterials
characterized by negative electric permittivity 
and/or negative magnetic permeability.
\end{abstract}

\maketitle

\section{Introduction}
Consider the domains $\Omega_+=(0,1)\times (0,1)$ and $\Omega_-=(-1,0)\times(0,1)$ and let $\Omega=(-1,1)\times (0,1)$ and $\cC=\{0\}\times (0,1)$.
We study the nonelliptic differential expression $\cA$ defined by
\begin{equation}\label{A}
 \cA f=-\text{div}\, (\sgn\nabla f\,),\qquad \text{where}\quad \sgn(x,y)=\begin{cases} 1, & (x,y)\in\Omega_+, \\ -1, & (x,y)\in\Omega_-, \end{cases}
\end{equation}
on the rectangle $\Omega$. Our aim is to associate a selfadjoint operator in $L^2(\Omega)$ with Dirichlet boundary conditions on $\partial\Omega$ to $\cA$. 
Informally speaking, in this seemingly simple toy problem this will be the partial differential operator
\begin{equation}\label{opaopa}
 \begin{split}
  A f&=\cA f=\begin{pmatrix} -\Delta f_+\\ \Delta f_-\end{pmatrix},\\
 \dom A&=\left\{f=\begin{pmatrix}f_+\\ f_-\end{pmatrix}: \begin{matrix} f_\pm,\,\Delta f_\pm\in L^2(\Omega_\pm),\, f\vert_{\partial\Omega}=0,\\ 
        f_+\vert_\cC=f_-\vert_\cC,\, \partial_{{\bf n}_+}f_+\vert_\cC=
 \partial_{{\bf n}_-}f_-\vert_\cC                                                  \end{matrix}\right\},
 \end{split}
\end{equation}
where $f_\pm$ denote the restrictions of a function $f\in L^2(\Omega)$ onto $\Omega_\pm$, and
the normal derivatives $\partial_{{\bf n}_+}$ and $\partial_{{\bf n}_-}$ point outward of $\Omega_\pm$ 
(and hence in opposite directions at $\cC$).
The main peculiarity here is the interface condition 
\begin{equation*}
 \partial_{{\bf n}_+}f_+\vert_\cC=\partial_{{\bf n}_-}f_-\vert_\cC,
 \qquad f=(f_+,f_-)^\top\in\dom A,      
\end{equation*}
for the normal derivatives, which is due to the sign change 
and discontinuity
of the coefficient $\sgn$ at  $\cC$.
Our main result states that (when the Dirichlet and Neumann traces are properly interpreted) the operator $A$ in \eqref{opaopa} is selfadjoint 
in $L^2(\Omega)$. 

The non-standard interface condition is responsible
for unexpected spectral properties of~$A$.
Although the domain~$\Omega$ is bounded,
it turns out that the essential spectrum of~$A$ is not empty, namely
$0$~is an isolated eigenvalue of infinite multiplicity. The remaining part of the spectrum of $A$ consists of discrete eigenvalues
which accumulate to $+\infty$ and $-\infty$. 
We note that the differential equation
$\cA f=\lambda f$ can of course be solved by separation of variables; 
the main feature
of this note is the description of the domain of the corresponding selfadjoint operator~$A$ with explicit boundary and interface conditions.

We point out that $\dom A$ contains functions which do not belong to any
local Sobolev space $H^s$, $s>0$, in a neighbourhood of the interface $\cC$. This leads to the following difficulties:
Green's identity is not valid for functions $f,g\in\dom A$ 
and the definition of the (local) Dirichlet and Neumann traces is rather subtle,
and requires a particularly careful analysis. Here we employ recent results on the extension of trace maps onto maximal domains of Laplacians
on (quasi-)convex and Lipschitz domains from \cite{BM14,GM11} and we rely on the description of the traces of $H^2(\Omega_\pm)$-functions in \cite{GK00}.
It finally turns out that the operator $A$ can be viewed as a kind of Krein-von Neumann extension of a non-semibounded
symmetric operator with infinite defect and domain contained in  
$H^2(\Omega_+)\times H^2(\Omega_-)$; 
thus only the functions in the infinite dimensional
eigenspace $\ker A$ do not possess 
$H^s$-regularity near the interface $\cC$.

We wish to emphasize that our result complements the results in \cite{BDR99} where 
the related problem
\begin{equation}\label{A_eps}
 \cA_\varepsilon f=-\text{div}\, (\varepsilon\,\nabla f),\qquad \quad \varepsilon(x,y)=\begin{cases} \varepsilon_+, & (x,y)\in\Omega_+, \\ -\varepsilon_-, & (x,y)\in\Omega_-, 
 \end{cases}
\end{equation}
with $\varepsilon_\pm>0$ was treated under the assumption $\varepsilon_+\not=\varepsilon_-$ with the help of boundary integral methods on more general
domains $\Omega\subset\dR^2$; 
for related problems see also
\cite{Bonnet_2012,Costabel-Stephan_1985,DT97,Grieser,Hussein-thesis,Hussein}.
It is shown in \cite{BDR99} that, if $\varepsilon_+\not=\varepsilon_-$, the operator
\begin{equation}\label{A_eps.op}
  A_\varepsilon f = \cA_\varepsilon f, 
  \qquad
  \dom A_\varepsilon 
  = \bigl\{f \in H_0^1(\Omega) : \cA_\varepsilon f \in L^2(\Omega)\bigr\},
\end{equation}
is selfadjoint, has a compact resolvent, 
and with eigenvalues accumulating to $+\infty$ and $-\infty$.
The borderline case $\varepsilon_+=\varepsilon_-$ that we investigate in this note was excluded in \cite{BDR99} and the other works
(except for the one-dimensional situation~\cite{Hussein}, 
which is intrinsically different). We also wish to mention that abstract representation theorems for indefinite quadratic forms and related
form methods in \cite{GKMV10} (see also \cite{FHS00,Hussein-thesis} 
and \cite{PTW14,TW14}) 
are not directly applicable in the present problem or do not lead to a selfadjoint operator in $L^2(\Omega)$.
The eigenvalue problem $\cA_\varepsilon f = \lambda f$ 
in our rectangular geometry
was previously considered in~\cite{Hussein-thesis}
with the help of separation of variables (cf.~Section~\ref{Sec.spec}),
from which it follows that~$0$ is an eigenvalue of
infinite multiplicity provided that $\varepsilon_+=\varepsilon_-$.

The indefinite differential expressions~\eqref{A} and~\eqref{A_eps} 
arise in mathematical models of metamaterials 
which are characterized by negative electric permittivity 
and/or negative magnetic permeability
(see \cite{Pendry_2004,Smith} for a physical survey 
and \cite{BBF,BS,FB} for a rigorous justification
of the models via a homogenization of Maxwell's equations
in geometrically non-trivial periodic structures). 
More specifically, our rectangular model can be thought 
as simulating an interface between 
a dielectric material in $\Omega_+$ and a metamaterial in $\Omega_-$.
It has been known since the seminal work~\cite{Costabel-Stephan_1985}
that the problem of the type $\cA_\varepsilon f=\rho$ 
in~$\Omega$ with a smooth interface
is well posed in $H_0^1(\Omega)$ if and only if
the contrast $\kappa:=\varepsilon_+/\varepsilon_-$
is different from~$1$.
Proving that~\eqref{opaopa} is selfadjoint,
in this note we provide a correct functional setting
for the problem on a rectangle in the critical situation $\kappa=1$.
Moreover, in Section~\ref{Sec.spec} of this note we show that the
eigenvalues and eigenfunctions of~$A_\varepsilon$ converge
to eigenvalues and eigenfunctions of the operator~$A$
as $\kappa \to 1$.

An alternative approach to theoretical studies of metamaterials
is to add a small imaginary number to the negative value 
of~$\mathrm{sgn}$,
arguing that ``real systems are always slightly lossy'', see, e.g. \cite{Pendry_2004}.
This leads to a complexified differential expression
\begin{equation}\label{A.complex}
 \cB_\eta f=-\text{div}\, (\varepsilon_\eta\nabla f),
  \qquad \quad 
  \varepsilon_\eta(x,y)=\begin{cases} 1, & 
  (x,y)\in\Omega_+, \\ -1 + i\eta, & (x,y)\in\Omega_-, 
 \end{cases}
\end{equation}
with $\eta > 0$, which immediately provides a well-defined operator 
\begin{equation}\label{A.complex.op}
  B_\eta f = \cB_\eta f, 
  \qquad
  \dom B_\eta = \bigl\{f \in H_0^1(\Omega) : \cB_\eta f \in L^2(\Omega)\bigr\}.
\end{equation}
Indeed, the rotated operator $e^{-i(\pi/2-\eta)} B_\eta$ 
is an $m$-sectorial operator with vertex~$0$ and semi-angle~$\pi/2-\eta$,
which is defined via the associated sectorial form defined on $H_0^1(\Omega)$; cf.~\cite[Sec.~VI]{K80}.
It follows that~$B_\eta$ is an operator 
with compact resolvent for every $\eta>0$,
albeit non-selfadjoint now.
Let us note that considering the complexified problem $B_\eta f = \rho$
in the limit as $\eta \to 0$ is a conventional way how to describe 
the cloaking effects in metamaterials (of different geometric structure)
through the ``anomalous localized resonance'', see \cite{Bouchitte, Milton}.
We shall show that the eigenvalues and eigenfunctions of~$B_\eta$ converge
to eigenvalues and eigenfunctions of our operator~$A$ as $\eta \to 0$.
Recall that $A f = \rho$ is generally ill-posed since~$0$ is an eigenvalue 
of infinite multiplicity .

This note is organized as follows.
In Section~\ref{Sec.Green} we establish
a modified version of Green's identity
and other preliminary results that we shall frequently use later. 
In Section~\ref{Sec.aux} we introduce an auxiliary closed
symmetric operator~$R$ and study its properties. 
By considering a generalized Krein--von Neumann extension of $R$,
the selfadjointness of~$A$ is proved in Section~\ref{Sec.ss}, 
where we also discuss qualitative spectral properties of~$A$. 
More quantitative results about the spectrum of~$A$
and the aforementioned convergence results 
are established in Section~\ref{Sec.spec}.

\subsection*{Acknowledgement.} 
We wish to thank our colleagues 
Guy Bouchitt\'e, E.~Brian Davies, Amru Hussein, Vadim Kostrykin,
Rainer Picard, Karl-Michael Schmidt, and Sascha Trostorff for fruitful discussions.
This work is supported by the Austrian Science Fund (FWF), project P 25162-N26,
Czech project RVO61389005 and the GACR grant No.\ 14-06818S.

\section{A generalized Green's identity on the maximal domain}\label{Sec.Green}

The Dirichlet realizations $A_{D\pm}$ associated to $\mp\Delta $ in $L^2(\Omega_\pm)$ will play an important role in the sequel. Recall that
\begin{equation}
 A_{D\pm}=\mp \Delta,\qquad \dom A_{D\pm}=H^1_0(\Omega_\pm)\cap H^2(\Omega_\pm),
\end{equation}
are selfadjoint operators in $L^2(\Omega_\pm)$ with compact resolvents, that $A_{D+}$ is uniformly positive, and that $A_{D-}$ is uniformly negative.
Here the $H^2$-regularity is consequence of $\Omega_\pm$ being convex; 
cf.~\cite{G85,G92}. If $\gamma_D$ denotes the Dirichlet trace operator
defined on $H^2(\Omega_\pm)$ then one has 
\begin{equation*}
 \dom A_{D\pm}=\bigl\{f_\pm\in H^2(\Omega_\pm):\gamma_D f_\pm=0\bigr\}.
\end{equation*}

The selfadjoint Neumann operators are given by 
\begin{equation*}
A_{N\pm}=\mp \Delta,\qquad \dom A_{N\pm}=\bigl\{f_\pm\in H^2(\Omega_\pm):\gamma_{N_\pm} f_\pm=0\bigr\},
\end{equation*}
where $\gamma_{N_\pm}$ are the Neumann trace operator defined on $H^2(\Omega_\pm)$ with normal pointing outwards $\Omega_\pm$. 

We shall also make use of the spaces
\begin{equation*}
\begin{split}
 \mathscr G_N(\partial\Omega_\pm)&:=\ran\bigl(\gamma_{N_\pm}(\dom A_{D\pm})\bigr)=\bigl\{\gamma_{N_\pm} f_\pm: f_\pm\in H^2(\Omega_\pm),\,
 \gamma_D f_\pm=0\bigr\},\\
 \mathscr G_D(\partial\Omega_\pm)&:=\ran\bigl(\gamma_D(\dom A_{N\pm})\bigr)
 =\bigl\{\gamma_D f_\pm: f_\pm\in H^2(\Omega_\pm),\,\gamma_{N_\pm} f_\pm
 =0\bigr\},
 \end{split}
 \end{equation*}
which were characterized and denoted by $N^{1/2}(\partial\Omega_\pm)$ and $N^{3/2}(\partial\Omega_\pm)$, respectively, in \cite{GM11}, and also appear 
in \cite{BM14} in a more general setting. We equip $\mathscr G_N(\partial\Omega_\pm)$ and $\mathscr G_D(\partial\Omega_\pm)$ with the natural norms 
\cite[(6.6) and (6.42)]{GM11}.
If ${\bf n}_\pm$ and ${\bf t}_\pm$ denote the unit normal pointing outwards and a corresponding tangential vector, respectively,  
and $\partial_{{\bf t}_\pm}$ is the tangential derivative on $\partial\Omega_\pm$, then according to
\cite[Theorem 3]{GK00} one has
\begin{equation*}
 \bigl(\gamma_{N_\pm}f_\pm\bigr){\bf t}_\pm \in \bigl(H^{1/2}(\partial\Omega_\pm)\bigr)^2
\end{equation*}
 for all $\gamma_{N_\pm}f_\pm\in\mathscr G_N(\partial\Omega_\pm)$ and
 \begin{equation*}
 \bigl(\partial_{{\bf t}_\pm}\gamma_D f_\pm \bigr){\bf n}_\pm \in \bigl(H^{1/2}(\partial\Omega_\pm)\bigr)^2
 \end{equation*}
 for all $\gamma_D f_\pm\in\mathscr G_D(\partial\Omega_\pm)$, where
 \begin{equation*}
  H^{1/2}(\partial\Omega_\pm)=\left\{\varphi\in L^2(\partial\Omega_\pm):\int_{\partial\Omega_\pm}\int_{\partial\Omega_\pm}
  \frac{\vert \varphi(\alpha)-\varphi(\beta)\vert^2}{\vert\alpha-\beta\vert^2}\,d\alpha\,d\beta<\infty\right\}.
 \end{equation*}
The following statement on the decomposition of functions in $\mathscr G_N(\partial\Omega_\pm)$ and $\mathscr G_D(\partial\Omega_\pm)$
in two parts with supports on $\cC$ and $\cC_\pm:=\partial\Omega_\pm\backslash\cC$, respectively, is a direct consequence of the abovementioned
fact.

\begin{lemma}\label{glem} 
Every function $\varphi\in \mathscr G_N(\partial\Omega_\pm)$
 (resp. $\varphi\in\mathscr G_D(\partial\Omega_\pm)$)
 admits a decomposition in the form
 \begin{equation}
  \varphi=(\varphi\vert_\cC)^\sim + (\varphi\vert_{\cC_\pm})^\sim 
 \end{equation}
where $(\varphi\vert_\cC)^\sim\in\mathscr G_N(\partial\Omega_\pm)$ (resp. $(\varphi\vert_\cC)^\sim\in\mathscr G_D(\partial\Omega_\pm)$) is the extension of $\varphi\vert_\cC$ to $\partial\Omega_\pm$ by $0$, and 
$(\varphi\vert_{\cC_\pm})^\sim\in\mathscr G_N(\partial\Omega_\pm)$ (resp. $(\varphi\vert_{\cC_\pm})^\sim\in\mathscr G_D(\partial\Omega_\pm)$) is the extension of $\varphi\vert_{\cC_\pm}$ to $\partial\Omega_\pm$ by $0$.
\end{lemma}

Consider the symmetric operators $S_\pm=\mp\Delta$, 
$\dom S_\pm=H^2_0(\Omega_\pm)$, and their adjoints
\begin{equation}\label{maxops}
 S_\pm^*=\mp\Delta,\qquad \dom S_\pm^*=\bigl\{f_\pm\in L^2(\Omega_\pm): \Delta f_\pm\in L^2(\Omega_\pm)\bigr\}.
\end{equation}
Since $0\not\in\sigma(A_{D,\pm})$ one has the direct sum decompositions
\begin{equation}\label{deco1}
 \dom S_\pm^*=\dom A_{D\pm}\,\dot +\,\ker S_\pm^*.
\end{equation}
In the following we will often decompose functions $f_\pm\in\dom S_\pm^*$ accordingly, that is, we write
\begin{equation}\label{deco2}
 f_\pm= f_{D\pm}+ f_{0\pm},\qquad f_{D\pm}\in\dom A_{D\pm},\,\, f_{0\pm}\in \ker S_\pm^*. 
\end{equation}
It is also important to note that the spaces $\ker S_\pm^*\cap H^2(\Omega_\pm)$ are dense 
in $\ker S_\pm^*$, where the latter spaces are equipped with the $L^2$-norm 
(or, equivalently with the graph norm of $S_\pm^*$). This fact can be shown with the help of the density result \cite[(6.30)]{GM11} 
for $s=0$.

Recall from \cite[Theorem 6.4]{GM11} that the Dirichlet traces $\gamma_D$ admit continuous and surjective extensions
\begin{equation*}
 \widetilde\gamma_D:\dom S_\pm^* \rightarrow \bigl(\mathscr G_N(\partial\Omega_\pm)\bigr)^*,
\end{equation*}
where $\dom S_\pm^*$ is equipped with the graph norm and $(\mathscr G_N(\partial\Omega_\pm))^*$ is the conjugate dual space of $\mathscr G_N(\partial\Omega_\pm)$
equipped with the corresponding norm.
It is important to note that
\begin{equation}\label{kerndd}
 \ker\widetilde\gamma_D =\ker\gamma_D = \dom A_{D\pm}=H^1_0(\Omega_\pm)\cap H^2(\Omega_\pm),
\end{equation}
where the first equality has been shown in \cite[Section 4.1]{BM14} and the other identities are clear from the above.

We shall denote the duality pairing between $\mathscr G_N(\partial\Omega_\pm)$ and $(\mathscr G_N(\partial\Omega_\pm))^*$ 
in the form
\begin{equation*}
 {}_{\mathscr{G}_N(\partial\Omega_\pm)^*}\langle\psi,\varphi\rangle_{\mathscr{G}_N(\partial\Omega_\pm)},\qquad 
 \psi\in\mathscr{G}_N(\partial\Omega_\pm)^*,\quad\varphi\in\mathscr{G}_N(\partial\Omega_\pm),
\end{equation*}
and occasionally we also write $\psi(\varphi)$ in this situation.

It will also be used later that the Neumann traces $\gamma_{N_\pm}$ admit continuous and surjective extensions
\begin{equation*}
 \widetilde\gamma_{N_\pm}:\dom S_\pm^* \rightarrow \bigl(\mathscr G_D(\partial\Omega_\pm)\bigr)^*;
\end{equation*}
this fact was observed in \cite[Theorem 6.10]{GM11}. Here again 
$\dom S_\pm^*$ is equipped with the graph norm and $(\mathscr G_D(\partial\Omega_\pm))^*$ is the conjugate dual space of $\mathscr G_D(\partial\Omega_\pm)$
equipped with the corresponding norm.

The next proposition shows that a modified Green's identity (with the Neumann trace $\gamma_{N_\pm} f_\pm$ replaced by the regularized
Neumann trace $\gamma_{N_\pm} f_{D\pm}$) remains valid on the maximal domains $\dom S_\pm^*$. This fact is essentially a consequence of
\cite[Theorem~6.4]{GM11}. We also mention that analogous extensions of Green's identity are well known for elliptic operators on
smooth domains, see, e.g. \cite{G68}.

\begin{proposition}\label{green}
The following Green's identity holds for all $f_\pm=f_{D\pm}+f_{0\pm}$ and $g_\pm=g_{D\pm}+g_{0\pm}$ in $\dom S_\pm^*$:
 \begin{equation*}
\begin{split}
&\bigl(S_\pm^* f_\pm,g_\pm\bigr)_{L^2(\Omega_\pm)}-\bigl(f_\pm,S_\pm^*g_\pm\bigr)_{L^2(\Omega_\pm)}\\
&\quad = \pm {}_{\mathscr{G}_N(\partial\Omega_\pm)^*}\big\langle\widetilde\gamma_D f_\pm,
\gamma_{N_\pm} g_{D\pm}\big\rangle_{\mathscr{G}_N(\partial\Omega_\pm)}\mp {}_{\mathscr{G}_N(\partial\Omega_\pm)}\big\langle\gamma_{N_\pm} f_{D\pm},
\widetilde \gamma_D g_\pm\big\rangle_{\mathscr{G}_N(\partial\Omega_\pm)^*}.
\end{split}
\end{equation*}
\end{proposition}
\begin{proof}
The identity will only be shown in $L^2(\Omega_+)$. The same argument applies on $\Omega_-$. 
Let $f_+=f_{D+}+f_{0+},\, g_+=g_{D+}+g_{0+}\in \dom S_+^*$ and recall from \cite[Theorem~6.4]{GM11} that
the identity 
\begin{equation*}
\bigl(S_+^* f_+,g_{D+}\bigr)-\bigl(f_+,A_{D+} g_{D+}\bigr)=
{}_{\mathscr{G}_N(\partial\Omega_+)^*}\big\langle\widetilde\gamma_D f_+,
\gamma_{N_+} g_{D+}\big\rangle_{\mathscr{G}_N(\partial\Omega_+)}
\end{equation*}
holds, where we simply write $(\cdot,\cdot)$ 
for the inner product in $L^2(\Omega_+)$. 
Since $A_{D+}$ is selfadjoint in $L^2(\Omega_+)$ it is clear that 
\begin{equation*}
 \bigl(A_{D+}f_{D+},g_{D+}\bigr)-\bigl(f_{D+},A_{D+}g_{D+}\bigr)=0.
\end{equation*}
Moreover, as $f_{0+},g_{0+}\in\ker S_+^*$ we also have
\begin{equation*}
 \bigl(S_+^* f_{0+},g_{0+}\bigr)-\bigl(f_{0+},S_+^* g_{0+}\bigr)=0.
\end{equation*}
Taking this into account we compute
\begin{equation*}
\begin{split}
&(S_+^* f_+,g_+)-(f_+,S_+^*g_+)\\
&\quad =\bigl(S_+^* (f_{D+}+f_{0+}),g_{D+}+g_{0+}\bigr)-\bigl(f_{D+}+f_{0+},S_+^*(g_{D+}+g_{0+})\bigr)\\
&\quad =\bigl(A_{D+}f_{D+},g_{0+}\bigr) + \bigl(S_+^*f_{0+}, g_{D+}\bigr) 
-\bigl(f_{0+},A_{D+}g_{D+}\bigr)-\bigl(f_{D+},S_+^*g_{0+}\bigr)\\
&\quad =\bigl(A_{D+}f_{D+},g_{0+}\bigr)-\bigl(f_{D+},S_+^*g_{0+}\bigr) + \bigl(S_+^*f_{0+}, g_{D+}\bigr)-\bigl(f_{0+},A_{D+}g_{D+}\bigr)\\
&\quad =  - {}_{\mathscr{G}_N(\partial\Omega_+)}\big\langle\gamma_{N_+} f_{D+},
\widetilde \gamma_D g_{0+}\big\rangle_{\mathscr{G}_N(\partial\Omega_+)^*}+
{}_{\mathscr{G}_N(\partial\Omega_+)^*}\big\langle\widetilde\gamma_D f_{0+},
\gamma_{N_+} g_{D+}\big\rangle_{\mathscr{G}_N(\partial\Omega_+)}\\
&\quad = {}_{\mathscr{G}_N(\partial\Omega_+)^*}\big\langle\widetilde\gamma_D f_+,
\gamma_{N_+} g_{D+}\big\rangle_{\mathscr{G}_N(\partial\Omega_+)}- {}_{\mathscr{G}_N(\partial\Omega_+)}\big\langle\gamma_{N_+} f_{D+},
\widetilde \gamma_D g_+\big\rangle_{\mathscr{G}_N(\partial\Omega_+)^*},
\end{split}
\end{equation*}
where we have used $\ker\widetilde\gamma_D=\ker\gamma_D$ from \eqref{kerndd} in the last identity. 
\end{proof}

Next we consider the subspaces
\begin{equation*}
 \mathscr G_\pm:=\bigl\{\varphi\in\mathscr G_N(\partial\Omega_\pm): \varphi\vert_\cC=0 \bigr\}
\end{equation*}
of $\mathscr G_N(\partial\Omega_\pm)$ which consist of functions vanishing on $\cC$.
Denote by $\mathscr G_\pm^\bot\subset (\mathscr G_N(\partial\Omega_\pm))^*$ the corresponding annihilators,
\begin{equation*}
 \mathscr G_\pm^\bot=\bigl\{\psi\in(\mathscr G_N(\partial\Omega_\pm))^*:\psi(\varphi)=0\,\,\text{for all}\,\,\varphi\in\mathscr G_\pm \bigr\}.
\end{equation*}
Roughly speaking $\mathscr G_\pm^\bot$ can be viewed as the linear subspaces of functionals from  $(\mathscr G_N(\partial\Omega_\pm))^*$
that vanish on $\cC_\pm=\partial\Omega_\pm\backslash\cC$. It is important to note that
\begin{equation}\label{anni}
 \mathscr G_\pm^\bot\cong \bigl(\mathscr G_N(\partial\Omega_\pm) / \mathscr G_\pm \bigr)^*.
\end{equation}
In particular, if for some $\varphi\in\mathscr G_N(\partial\Omega_\pm)$ and all $\psi\in \mathscr G_\pm^\bot$ one has $\psi(\varphi)=0$
then $\varphi=0$ when identified with elements in the quotient space $\mathscr G_N(\partial\Omega_\pm) / \mathscr G_\pm $ and
hence $\varphi\in\mathscr G_\pm$, that is, $\varphi\vert_\cC=0$.

\section{An auxiliary symmetric operator $R$}\label{Sec.aux}
In the next proposition we consider a restriction $R$ of the selfadjoint operator $ A_{D+}\oplus A_{D-}$ in $L^2(\Omega)$
and we determine the adjoint of $R$. It will later turn out that the operator $A$ in \eqref{opaopa} is a selfadjoint extension of $R$ (and hence a 
restriction of the adjoint operator  $R^*$).
\begin{proposition}\label{rrprop}
 The operator
 \begin{equation*}
 \begin{split}
  R f&=\cA f=\begin{pmatrix} -\Delta f_+\\ \Delta f_-\end{pmatrix},\\
 \dom R&=\left\{f=\begin{pmatrix}f_+\\ f_-\end{pmatrix}:  f_\pm\in H^2(\Omega_\pm)\cap H^1_0(\Omega_\pm),\, \gamma_{N_+} f_+\vert_\cC=\gamma_{N_-} f_-\vert_\cC\right\},
 \end{split}
\end{equation*}
 is a closed symmetric operator with equal infinite deficiency indices in $L^2(\Omega)$ and $R\subset A_{D+}\oplus A_{D-}$ holds. 
 The adjoint operator is given by
 \begin{equation*}
 \begin{split}
  R^* f&=\cA f=\begin{pmatrix} -\Delta f_+\\ \Delta f_-\end{pmatrix},\\
 \dom R^*&=\left\{f=\begin{pmatrix}f_+\\ f_-\end{pmatrix}: f_\pm,\,\Delta f_\pm\in L^2(\Omega_\pm),\,
 \widetilde\gamma_D f_\pm\in \mathscr G_\pm^\bot,\, \widetilde\gamma_D f_+\vert_\cC=\widetilde\gamma_D f_-\vert_\cC\right\}, 
 \end{split}
\end{equation*}
where the boundary condition $\widetilde\gamma_D f_+\vert_\cC=\widetilde\gamma_D f_-\vert_\cC$ is understood as 
\begin{equation*}
 {}_{\mathscr{G}_N(\partial\Omega_+)^*}\big\langle\widetilde\gamma_D f_+,\varphi\big\rangle_{\mathscr{G}_N(\partial\Omega_+)}=
 {}_{\mathscr{G}_N(\partial\Omega_-)^*}\big\langle\widetilde\gamma_D f_-,\varphi\big\rangle_{\mathscr{G}_N(\partial\Omega_-)}
\end{equation*}
for all $\varphi\in\mathscr G_N(\partial\Omega_\pm)$ such that $\varphi\vert_{\cC_\pm}=0$.
\end{proposition}

\begin{proof}
The proof consists of three steps. We define the operator
 \begin{equation*}
 \begin{split}
  T f&:=\cA f=\begin{pmatrix} -\Delta f_+\\ \Delta f_-\end{pmatrix},\\
 \dom T&:=\left\{f=\begin{pmatrix}f_+\\ f_-\end{pmatrix}: f_\pm,\,\Delta f_\pm\in L^2(\Omega_\pm),\,
 \widetilde\gamma_D f_\pm\in \mathscr G_\pm^\bot,\, \widetilde\gamma_D f_+\vert_\cC=\widetilde\gamma_D f_-\vert_\cC\right\},
 \end{split}
\end{equation*}
and it will be shown in Step 1 and Step 2 that $T^*=R$. In Step 3 we verify that $T$ is closed, so that
\begin{equation*}
 R^*=T^{**}=\overline T=T.
\end{equation*}

\vskip 0.15cm\noindent
{\it Step 1.} We verify that $R\subset T^*$ holds. 
For this fix some $f=(f_+,f_-)^\top\in\dom R$, and note that
$f_\pm=f_{D\pm}$ in the decomposition \eqref{deco1}--\eqref{deco2}. 
As both $T$ and $R$ are restrictions of the 
orthogonal sum $S_+^*\oplus S_-^*$ of the maximal operators in \eqref{maxops} it follows from Green's identity in Proposition~\ref{green}
that for any $g\in\dom T$ decomposed in the form
$g_\pm=g_{D\pm}+g_{0\pm}$ we have
\begin{equation*}
\begin{split}
 &\bigl(Rf,g\bigr)_{L^2(\Omega)}-\bigl(f,Tg\bigr)_{L^2(\Omega)}=\bigl((S_+^*\oplus S_-^*)f,g\bigr)_{L^2(\Omega)}-\bigl(f,(S_+^*\oplus S_-^*)g\bigr)_{L^2(\Omega)}\\
 &\quad =\bigl(S_+^* f_+,g_+\bigr)_{L^2(\Omega_+)}-\bigl(f_+,S_+^*g_+\bigr)_{L^2(\Omega_+)} \\
 &\qquad\qquad\qquad\qquad + \bigl(S_-^* f_-,g_-\bigr)_{L^2(\Omega_-)}
 -\bigl(f_-,S_-^*g_-\bigr)_{L^2(\Omega_-)}\\
& \quad = {}_{\mathscr{G}_N(\partial\Omega_+)^*}\big\langle\widetilde\gamma_D f_+,
\gamma_{N_+} g_{D+}\big\rangle_{\mathscr{G}_N(\partial\Omega_+)}- {}_{\mathscr{G}_N(\partial\Omega_+)}\big\langle\gamma_{N_+} f_{D+},
\widetilde \gamma_D g_+\big\rangle_{\mathscr{G}_N(\partial\Omega_+)^*}\\
& \quad\qquad - {}_{\mathscr{G}_N(\partial\Omega_-)^*}\big\langle\widetilde\gamma_D f_-,
\gamma_{N_-} g_{D-}\big\rangle_{\mathscr{G}_N(\partial\Omega_-)} +  {}_{\mathscr{G}_N(\partial\Omega_-)}\big\langle\gamma_{N_-} f_{D-},
\widetilde \gamma_D g_-\big\rangle_{\mathscr{G}_N(\partial\Omega_-)^*}\\
&\quad = - {}_{\mathscr{G}_N(\partial\Omega_+)}\big\langle\gamma_{N_+} f_+,
\widetilde \gamma_D g_+\big\rangle_{\mathscr{G}_N(\partial\Omega_+)^*} +  {}_{\mathscr{G}_N(\partial\Omega_-)}\big\langle\gamma_{N_-} f_-,
\widetilde \gamma_D g_-\big\rangle_{\mathscr{G}_N(\partial\Omega_-)^*},
\end{split}
\end{equation*}
where in the last step we have used that 
for $f=f_+\oplus f_-\in\dom R$ one has $f_\pm=f_{D\pm}$, and $f_\pm\in H^1_0(\Omega_\pm)$, so that, $\widetilde\gamma_D f_\pm=0$; cf. \eqref{kerndd}.
Next we decompose  $\gamma_{N_\pm} f_\pm$ in the form
\begin{equation}\label{gammanfdeco}
 \gamma_{N_\pm} f_\pm=\bigl(\gamma_{N_\pm} f_\pm\vert_\cC\bigr)^\sim+ \bigl(\gamma_{N_\pm} f_\pm\vert_{\cC_\pm}\bigr)^\sim,
\end{equation}
where both extensions by $0$ on the right hand side belong to the space $\mathscr{G}_N(\partial\Omega_\pm)$ (see Lemma~\ref{glem}), and in particular
\begin{equation*}
 \bigl(\gamma_{N_\pm} f_{D\pm}\vert_{\cC_\pm}\bigr)^\sim\in\mathscr G_\pm.
\end{equation*}
Since  $g\in\dom T$ we have $\widetilde\gamma_D g_\pm\in\mathscr G_\pm^\bot$ and therefore
\begin{equation*}
 {}_{\mathscr{G}_N(\partial\Omega_\pm)}\big\langle\bigl(\gamma_{N_\pm} f_\pm\vert_{\cC_\pm}\bigr)^\sim,
\widetilde \gamma_D g_\pm\big\rangle_{\mathscr{G}_N(\partial\Omega_\pm)^*}=0.
\end{equation*}
Hence we conclude
\begin{equation}\label{randterm2}
 \begin{split}
 \bigl(Rf,g\bigr)_{L^2(\Omega)}-\bigl(f,Tg\bigr)_{L^2(\Omega)} &=- {}_{\mathscr{G}_N(\partial\Omega_+)}\big\langle\bigr(\gamma_{N_+} f_+\vert_\cC\bigr)^\sim,
\widetilde \gamma_D g_+\big\rangle_{\mathscr{G}_N(\partial\Omega_+)^*}\\
 &\qquad +
{}_{\mathscr{G}_N(\partial\Omega_-)}\big\langle\bigr(\gamma_{N_-} f_-\vert_\cC\bigr)^\sim,
\widetilde \gamma_D g_-\big\rangle_{\mathscr{G}_N(\partial\Omega_-)^*}.
\end{split}
\end{equation}
Since $f\in\dom R$ and $g\in\dom T$ we obtain 
\begin{equation*}
\gamma_{N_+} f_+\vert_\cC=\gamma_{N_-} f_-\vert_\cC\quad\text{and}\quad\widetilde\gamma_D g_+\vert_\cC=\widetilde\gamma_D g_-\vert_\cC.
\end{equation*}
This and \eqref{randterm2} implies that $(Rf,g)_{L^2(\Omega)}-(f,Tg)_{L^2(\Omega)}=0$ holds for all $g\in\dom T$. Therefore $f\in\dom T^*$ and $T^*f=Rf$.
We have shown $R\subset T^*$.

\vskip 0.15cm\noindent
{\it Step 2.} We now verify the opposite inclusion $T^*\subset R$. 
For this observe first that the orthogonal sum of the Dirichlet operator 
$A_{D+}\oplus A_{D-}$ is a selfadjoint restriction of $T$, 
and hence we have
\begin{equation}\label{ttt}
  T^* \subset A_{D+}\oplus A_{D-} \subset T. 
\end{equation}
Let  $f=(f_+,f_-)^\top\in\dom T^*$. Then $f_\pm\in H^2(\Omega_\pm)\cap H^1_0(\Omega_\pm)$ 
and $f_\pm=f_{D\pm}$ in the decomposition \eqref{deco1}--\eqref{deco2}. It remains to show that the boundary condition
\begin{equation}\label{bc}
 \gamma_{N_+} f_+\vert_\cC=\gamma_{N_-} f_-\vert_\cC
\end{equation}
is satisfied. For this note that by \eqref{kerndd} 
we also have $\widetilde\gamma_D f_\pm=0$. For $g\in\dom T$ we
obtain in the same way as in Step 1 of the proof that
\begin{equation*}
\begin{split}
0&=\bigl(T^*f,g\bigr)_{L^2(\Omega)} - \bigl(f,Tg\bigr)_{L^2(\Omega)}\\ 
&= - {}_{\mathscr{G}_N(\partial\Omega_+)}\big\langle\gamma_{N_+} f_+,
\widetilde \gamma_D g_+\big\rangle_{\mathscr{G}_N(\partial\Omega_+)^*} +  {}_{\mathscr{G}_N(\partial\Omega_-)}\big\langle\gamma_{N_-} f_-,
\widetilde \gamma_D g_-\big\rangle_{\mathscr{G}_N(\partial\Omega_-)^*}.
\end{split}
\end{equation*}
Next we decompose $\gamma_{N_\pm} f_\pm$ as in \eqref{gammanfdeco} and use that $\widetilde\gamma_D g_\pm\in\mathscr G_\pm^\bot$. As in Step 1 this leads to
\begin{equation*}
0= {}_{\mathscr{G}_N(\partial\Omega_+)}\big\langle\bigr(\gamma_{N_+} f_+\vert_\cC\bigr)^\sim,
\widetilde \gamma_D g_+\big\rangle_{\mathscr{G}_N(\partial\Omega_+)^*}-
{}_{\mathscr{G}_N(\partial\Omega_-)}\big\langle\bigr(\gamma_{N_-} f_-\vert_\cC\bigr)^\sim,
\widetilde \gamma_D g_-\big\rangle_{\mathscr{G}_N(\partial\Omega_-)^*}
\end{equation*}
for all $g=(g_+,g_-)^\top\in\dom T$. Furthermore, since $\widetilde \gamma_D g_+\vert_\cC=\widetilde \gamma_D g_-\vert_\cC$ we find
\begin{equation*}
0= {}_{\mathscr{G}_N(\partial\Omega_+)}\big\langle\bigr(\gamma_{N_+} f_+\vert_\cC\bigr)^\sim-\bigr(\gamma_{N_-} f_-\vert_\cC\bigr)^\sim,
\widetilde \gamma_D g_+\big\rangle_{\mathscr{G}_N(\partial\Omega_+)^*}.
\end{equation*}
This relation holds true for all $g=(g_+,g_-)^\top\in\dom T$, and hence for all elements $\psi=\widetilde \gamma_D g_+\in\mathscr G_+^\bot$.
Now it follows from \eqref{anni} and the observation below \eqref{anni} that the function
\begin{equation*}
 \bigr(\gamma_{N_+} f_+\vert_\cC\bigr)^\sim-\bigr(\gamma_{N_-} f_-\vert_\cC\bigr)^\sim
\end{equation*}
vanishes on $\cC$.
Thus the 
boundary condition \eqref{bc} is satisfied. We have shown $f\in\dom T$ and hence $R^*\subset T$.

\vskip 0.15cm\noindent
{\it Step 3.} We show that $T$ is closed. Let $(f_n)\subset\dom T$ such that $f_n\rightarrow f$ and $Tf_n\rightarrow h$ for some $f=(f_+,f_-)^\top$, 
$h=(h_+,h_-)^\top\in L^2(\Omega)$.
Since $T\subset S_+^*\oplus S_-^*$ and $S_+^*\oplus S_-^*$ is closed it follows that 
\begin{equation*}
f_\pm \in\dom S_\pm^*\qquad\text{and}\qquad S_\pm^*f_\pm=h_\pm.
\end{equation*}
Thus it remains to show that the boundary conditions
\begin{equation*}
 \widetilde\gamma_D f_\pm \in\mathscr G_\pm^\bot\quad\text{and}\quad \widetilde\gamma_D f_+\vert_\cC=\widetilde\gamma_D f_-\vert_\cC
\end{equation*}
hold. But this follows immediately since $f_{n\pm}\rightarrow f_\pm$ in the graph norm of $S_\pm^*$ and $\widetilde\gamma_D$ is continuous with 
respect to the graph norm, so that, $\widetilde \gamma_D f_{n\pm}\rightarrow \widetilde\gamma_D f_\pm$ in $(\mathscr G_N(\partial\Omega_\pm))^*$.
\end{proof}

The following lemma states that the Neumann traces of the functions from $\ker R^*$ coincide on $\cC$. This property is essentially
a consequence of the symmetry of the domain $\Omega$ and the function $\sgn(\cdot)$ with respect to the interface $\cC$. For completeness
we mention that the functions 
\begin{equation}\label{efs}
 f_{0,k}(x,y)=\begin{cases} \sinh (k\pi (1-x))\sin (k\pi y), & (x,y)\in\Omega_+,\\ \sinh (k\pi (1+x))\sin (k\pi y), & (x,y)\in\Omega_-, \end{cases}
 \quad k \in \mathbb{N} = \{1,2,\dots\},
\end{equation}
span a dense set in $\ker R^*$; cf. Proposition~\ref{Prop.spec}~(iv).

\begin{lemma}\label{rrlem}
Let $R$ and $R^*$ be as in Proposition~\ref{rrprop}. Then the following hold.
\begin{itemize}
 \item [{\rm (i)}] The space $\ker R^*$ is infinite dimensional and the functions $f_0\in\ker R^*$ satisfy 
\begin{equation}\label{interface}
\widetilde\gamma_{N_+} f_{0+}\vert_\cC=\widetilde\gamma_{N_-} f_{0-}\vert_\cC,
\end{equation}
that is, 
\begin{equation*}
 {}_{\mathscr{G}_D(\partial\Omega_+)^*}\big\langle\widetilde\gamma_{N_+} f_{0+},\varphi\big\rangle_{\mathscr{G}_D(\partial\Omega_+)}=
 {}_{\mathscr{G}_D(\partial\Omega_-)^*}\big\langle\widetilde\gamma_{N_-} f_{0-},\varphi\big\rangle_{\mathscr{G}_D(\partial\Omega_-)}
\end{equation*}
holds for all $\varphi\in\mathscr G_D(\partial\Omega_\pm)$ such that $\varphi\vert_{\cC_\pm}=0$;
\item [{\rm (ii)}] $R$ is invertible and has closed range. 
\end{itemize}
\end{lemma}

\begin{proof}
(i)
As $A_{D+}\oplus A_{D-}\subset R^*$ and $0\not\in\sigma(A_{D_\pm})$ we have the direct sum decomposition
\begin{equation*}
 \dom R^*= \dom\bigl( A_{D+}\oplus A_{D-}\bigr)\,\dot +\,\ker R^*.
\end{equation*}
Together with \eqref{kerndd} this yields that the mapping 
\begin{equation}\label{wtwt}
 \widetilde\Gamma_D:\ker R^*\rightarrow\mathscr G_N(\partial\Omega_+)\times \mathscr G_N(\partial\Omega_-),\quad
 f_0=\begin{pmatrix} f_{0+}\\ f_{0-}\end{pmatrix}\mapsto \begin{pmatrix}\widetilde\gamma_D f_{0+}\\ \widetilde\gamma_D f_{0-}\end{pmatrix},
\end{equation}
is invertible. Suppose now that $f_0=(f_{0+},f_{0-})^\top\in\ker R^*$ and assume, in addition, that $f_{0\pm}\in H^2(\Omega_\pm)$. 
Then $\Delta f_{0\pm}=0$ and the boundary conditions have the explicit form 
\begin{equation}
 \gamma_D f_{0\pm}\vert_{\cC_\pm}=0\qquad\text{and}\qquad \gamma_D f_{0+}\vert_\cC=\gamma_D f_{0-}\vert_\cC;
\end{equation}
here $\gamma_D$ is the Dirichlet trace operator defined on $H^2(\Omega_\pm)$. 
It follows that the function 
$$h(x,y):=f_{0+}(-x,y),\qquad x\in (-1,0),\,\,y\in(0,1),$$
belongs to $L^2(\Omega_-)$ and satisfies $\Delta h=0$ and $\gamma_D h\vert_\cC=\gamma_D f_{0+}\vert_\cC$ and $\gamma_D h\vert_{\cC_-}=0$.
Hence $(f_{0+},h)^\top\in\ker R^*$ but as the map $\widetilde\Gamma_D$ in \eqref{wtwt} is invertible 
we conclude $f_{0-}=h$. In particular, if $\gamma_{N_\pm}$ denotes the Neumann trace operator on $H^2(\Omega_\pm)$ we obtain 
$$\gamma_{N_-} f_{0-}\vert_\cC=\gamma_{N_-} h\vert_\cC=\gamma_{N_+} f_{0+}\vert_\cC.$$
As $\widetilde\gamma_{N_\pm}$ are extensions of $\gamma_{N_\pm}$ this yields 
\begin{equation*}
 {}_{\mathscr{G}_D(\partial\Omega_+)^*}\big\langle\widetilde\gamma_{N_+} f_{0+},\varphi\big\rangle_{\mathscr{G}_D(\partial\Omega_+)}=
 {}_{\mathscr{G}_D(\partial\Omega_-)^*}\big\langle\widetilde\gamma_{N_-} f_{0-},\varphi\big\rangle_{\mathscr{G}_D(\partial\Omega_-)}
\end{equation*}
for all $\varphi\in\mathscr G_D(\partial\Omega_\pm)$ such that $\varphi\vert_{\cC_\pm}=0$.
We have shown that any function $f_0\in\ker R^*$ with the additional property $f_{0\pm}\in H^2(\Omega_\pm)$ satisfies the condition \eqref{interface}.
The general case follows from $R^*\subset S_+^*\oplus S_-^*$, the fact that $\ker S_\pm^*\cap H^2(\Omega_\pm)$ is dense in $\ker S_\pm^*$
and the continuity of the extended Neumann trace maps $\widetilde\gamma_{N_\pm}$.

\vskip 0.15cm\noindent
(ii) Since $R\subset A_{D+}\oplus A_{D-}\subset R^*$ and $0\not\in\sigma(A_{D_\pm})$ it follows that $\ker R=\{0\}$. In order to see that $\ran R$ is closed
assume that $Rf_n\rightarrow g$, $n\rightarrow\infty$, for some $g\in L^2(\Omega)$. It is clear that also $(A_{D+}\oplus A_{D-})f_n\rightarrow g$, $n\rightarrow\infty$,
and from $0\not\in\sigma(A_{D_\pm})$ we conclude 
\begin{equation*}
 f_n\rightarrow f:=\bigl(A_{D+}^{-1}\oplus A_{D-}^{-1}\bigr) g,\qquad n\rightarrow\infty.
\end{equation*}
Since $R$ is closed
we find $f\in\dom R$ and $Rf=g$. 
\end{proof}

\section{The selfadjoint operator $A$ and its qualitative spectral properties}\label{Sec.ss}

In this section we present the main result of this note. The operator $A$ (informally written in \eqref{opaopa}) is now defined rigorously with explicit boundary
conditions as a restriction of the maximal operator $S_+^*\oplus S_-^*$. It is shown that $A$ is selfadjoint in $L^2(\Omega)$ and it turns out 
that $A$ can be viewed as a generalized Krein-von Neumann extension of the non-semibounded symmetric 
operator $R$ (see also Proposition~\ref{mainprop} below). 

\begin{theorem}\label{mainthm}
 The operator
 \begin{equation}\label{opa}
 \begin{split}
  A f&=\cA f=\begin{pmatrix} -\Delta f_+\\ \Delta f_-\end{pmatrix},\\
 \dom A&=\left\{f=\begin{pmatrix}f_+\\ f_-\end{pmatrix}: 
 \begin{matrix} 
 f_\pm,\,\Delta f_\pm\in L^2(\Omega_\pm),\, \widetilde\gamma_D f_\pm\in \mathscr G_\pm^\bot,\\
 \widetilde\gamma_D f_+\vert_\cC=\widetilde\gamma_D f_-\vert_\cC,\,\widetilde\gamma_{N_+} f_+\vert_\cC=\widetilde\gamma_{N_-} f_-\vert_\cC
 \end{matrix}
 \right\},
 \end{split}
\end{equation}
is selfadjoint in $L^2(\Omega)$ and coincides with the operator $R^*\upharpoonright \dom R\,\dot +\,\ker R^*$.
The boundary conditions $\widetilde\gamma_D f_+\vert_\cC=\widetilde\gamma_D f_-\vert_\cC$ and 
$\widetilde\gamma_{N_+} f_+\vert_\cC=\widetilde\gamma_{N_-} f_-\vert_\cC$
are understood as 
\begin{equation*}
 {}_{\mathscr{G}_N(\partial\Omega_+)^*}\big\langle\widetilde\gamma_D f_+,\varphi\big\rangle_{\mathscr{G}_N(\partial\Omega_+)}=
 {}_{\mathscr{G}_N(\partial\Omega_-)^*}\big\langle\widetilde\gamma_D f_-,\varphi\big\rangle_{\mathscr{G}_N(\partial\Omega_-)}
\end{equation*}
for all $\varphi\in\mathscr G_N(\partial\Omega_\pm)$ such that $\varphi\vert_{\cC_\pm}=0$, and
\begin{equation*}
 {}_{\mathscr{G}_D(\partial\Omega_+)^*}\big\langle\widetilde\gamma_N f_+,\psi\big\rangle_{\mathscr{G}_D(\partial\Omega_+)}=
 {}_{\mathscr{G}_D(\partial\Omega_-)^*}\big\langle\widetilde\gamma_N f_-,\psi\big\rangle_{\mathscr{G}_D(\partial\Omega_-)}
\end{equation*}
for all $\psi\in\mathscr G_D(\partial\Omega_\pm)$ such that $\psi\vert_{\cC_\pm}=0$, respectively.
\end{theorem}

\begin{proof}
We first show that $A\subset A^*$ holds.
Since $A\subset R^*\subset S_+^*\oplus S_-^*$  we have for $f,g\in\dom A$ 
decomposed in the usual form $f_\pm=f_{D\pm}+f_{0\pm}$, $g_\pm=g_{D\pm}+g_{0\pm}$ (see \eqref{deco1}--\eqref{deco2})  
\begin{equation*}
\begin{split}
 &\bigl(Af,g\bigr)_{L^2(\Omega)}-\bigl(f,Ag\bigr)_{L^2(\Omega)}=\bigl((S_+^*\oplus S_-^*)f,g\bigr)_{L^2(\Omega)}-\bigl(f,(S_+^*\oplus S_-^*)g\bigr)_{L^2(\Omega)}\\
& \quad = {}_{\mathscr{G}_{N}(\partial\Omega_+)^*}\big\langle\widetilde\gamma_D f_+,
\gamma_{N_+} g_{D+}\big\rangle_{\mathscr{G}_N(\partial\Omega_+)}- {}_{\mathscr{G}_N(\partial\Omega_+)}\big\langle\gamma_{N_+} f_{D+},
\widetilde \gamma_D g_+\big\rangle_{\mathscr{G}_N(\partial\Omega_+)^*}\\
& \quad\qquad - {}_{\mathscr{G}_N(\partial\Omega_-)^*}\big\langle\widetilde\gamma_D f_-,
\gamma_{N_-} g_{D-}\big\rangle_{\mathscr{G}_N(\partial\Omega_-)} +  {}_{\mathscr{G}_N(\partial\Omega_-)}\big\langle\gamma_{N_-} f_{D-},
\widetilde \gamma_D g_-\big\rangle_{\mathscr{G}_N(\partial\Omega_-)^*};
\end{split}
\end{equation*}
cf.~Proposition~\ref{green} and  Step 1 in the proof of Proposition~\ref{rrprop}. Taking into account 
$\widetilde\gamma_D f_\pm,\widetilde\gamma_D g_\pm\in\mathscr G_\pm^\bot$
and decomposing $\gamma_{N_\pm} f_{D\pm}$ and $\gamma_{N_\pm} g_{D\pm}$ in the form 
\begin{equation*}
\begin{split}
 \gamma_{N_\pm} f_{D\pm}&=\bigl(\gamma_{N_\pm} f_{D\pm}\vert_\cC\bigr)^\sim+ \bigl(\gamma_{N_\pm} f_{D\pm}\vert_{\cC_\pm}\bigr)^\sim,\\
 \gamma_{N_\pm} g_{D\pm}&=\bigl(\gamma_{N_\pm} g_{D\pm}\vert_\cC\bigr)^\sim+ \bigl(\gamma_{N_\pm} g_{D\pm}\vert_{\cC_\pm}\bigr)^\sim,
 \end{split}
 \end{equation*}
where the extensions by $0$ on the right hand side belong to the spaces $\mathscr{G}_N(\partial\Omega_\pm)$ by Lemma~\ref{glem}, we find that  
\begin{equation*}
\begin{split}
 \bigl(Af,g\bigr)_{L^2(\Omega)}-\bigl(f,Ag\bigr)_{L^2(\Omega)}
 & = {}_{\mathscr{G}_{N}(\partial\Omega_+)^*}\big\langle\widetilde\gamma_D f_+,\bigl(\gamma_{N_+} g_{D+}\vert_\cC\bigr)^\sim\big\rangle_{\mathscr{G}_N(\partial\Omega_+)}\\
 &\qquad  - {}_{\mathscr{G}_N(\partial\Omega_+)}\big\langle\bigl(\gamma_{N_+} f_{D+}\vert_\cC\bigr)^\sim,\widetilde\gamma_D g_+\big\rangle_{\mathscr{G}_N(\partial\Omega_+)^*}\\
& \qquad - {}_{\mathscr{G}_N(\partial\Omega_-)^*}\big\langle\widetilde\gamma_D f_-,\bigl(\gamma_{N_-} g_{D-}\vert_\cC\bigr)^\sim\big\rangle_{\mathscr{G}_N(\partial\Omega_-)} \\
 &\qquad  +  {}_{\mathscr{G}_N(\partial\Omega_-)}\big\langle\bigl(\gamma_{N_-} f_{D-}\vert_\cC\bigr)^\sim,
 \widetilde \gamma_D g_-\big\rangle_{\mathscr{G}_N(\partial\Omega_-)^*}.
\end{split}
\end{equation*}
As $f,g\in\dom A$ we also have 
\begin{equation*}
 \widetilde\gamma_D f_+\vert_\cC=\widetilde\gamma_D f_-\vert_\cC \quad\text{and}\quad\widetilde\gamma_D g_+\vert_\cC=\widetilde\gamma_D g_-\vert_\cC
\end{equation*}
and hence the terms on the right hand side simplify to
\begin{equation}\label{rand1}
\begin{split}
 &{}_{\mathscr{G}_{N}(\partial\Omega_+)^*}\big\langle\widetilde\gamma_D f_+,
      \bigl(\gamma_{N_+} g_{D+}\vert_\cC\bigr)^\sim-\bigl(\gamma_{N_-} g_{D-}\vert_\cC\bigr)^\sim\big\rangle_{\mathscr{G}_N(\partial\Omega_+)}\\
&\qquad\quad - {}_{\mathscr{G}_N(\partial\Omega_+)}\big\langle\bigl(\gamma_{N_+} f_{D+}\vert_\cC\bigr)^\sim-\bigl(\gamma_{N_-} f_{D-}\vert_\cC\bigr)^\sim,
      \widetilde\gamma_D g_+\big\rangle_{\mathscr{G}_N(\partial\Omega_+)^*}.
 \end{split}
 \end{equation}
According to Lemma~\ref{rrlem}~(ii) the functions $f_{0\pm},g_{0\pm}\in\ker R^*$ satisfy 
\begin{equation*}
\widetilde\gamma_{N_+} f_{0+}\vert_\cC=\widetilde\gamma_{N_-} f_{0-}\vert_\cC\quad\text{and}\quad
\widetilde\gamma_{N_+} g_{0+}\vert_\cC=\widetilde\gamma_{N_-} g_{0-}\vert_\cC.
\end{equation*}
Thus we have
\begin{equation*}
 \begin{split}
  0=\widetilde\gamma_{N_+} f_+\vert_\cC - \widetilde\gamma_{N_-} f_-\vert_\cC
  &=\widetilde\gamma_{N_+} (f_{D+} + f_{0+})\vert_\cC - \widetilde\gamma_{N_-} (f_{D-} + f_{0-})\vert_\cC\\
  &=\gamma_{N_+} f_{D+} \vert_\cC - \gamma_{N_-} f_{D-} \vert_\cC
 \end{split}
\end{equation*}
and
\begin{equation*}
 \begin{split}
  0=\widetilde\gamma_{N_+} g_+\vert_\cC - \widetilde\gamma_{N_-} g_-\vert_\cC
  &=\widetilde\gamma_{N_+} (g_{D+} + g_{0+})\vert_\cC - \widetilde\gamma_{N_-} (g_{D-} + g_{0-})\vert_\cC\\
  &=\gamma_{N_+} g_{D+} \vert_\cC - \gamma_{N_-} g_{D-} \vert_\cC,
 \end{split}
\end{equation*}
and hence the corresponding entries in \eqref{rand1} vanish, that is, 
\begin{equation*}
\bigl(Af,g\bigr)_{L^2(\Omega)}-\bigl(f,Ag\bigr)_{L^2(\Omega)}=0,\qquad f,g\in\dom A.
\end{equation*}
We have shown that $A\subset A^*$ holds. 

Next we verify that the operator 
\begin{equation*}
 R_0:=R^*\upharpoonright\dom R\,\dot +\,\ker R^*
\end{equation*}
is contained in $A$. In fact, the inclusion $\dom R\subset \dom A$ is obvious and hence
it remains to show that $\ker R^*\subset\dom A$. It is clear from the definition of $\dom R^*$ that any 
function in $f_0=(f_{0+},f_{0-})\in\ker R^*$ satisfies the boundary conditions for functions in $\dom A$, 
with the exception of the condition
$\widetilde\gamma_{N_+}f_{0+}\vert_\cC=\widetilde\gamma_{N_-}f_{0-}\vert_\cC$. But this last condition holds by Lemma~\ref{rrlem}~(i).
Therefore $R_0\subset A$. We claim that $R_0$ is selfadjoint. First of all $R_0$ is symmetric since for $f=f_R+f_0\in\dom R\,\dot + \,\ker R^*$
one has
\begin{equation*}
 (R_0 f,f)_{L^2(\Omega)}=\bigl(R_0(f_R+f_0),f_R+f_0\bigr)_{L^2(\Omega)}=(R f_R,f_R)_{L^2(\Omega)},
\end{equation*}
and $R$ is symmetric. Moreover, by Lemma~\ref{rrlem}~(ii)  $0$ is a point of regular type of $R$, that is,
\begin{equation*}
 \ker R=\{0\}\quad\text{and}\quad\ran R\quad\text{is closed}.
\end{equation*}
This leads to the direct sum decomposition 
\begin{equation*}
\ran (R_0-\mu)=\ran(R-\mu)\, \dot+\, \ker R^*=L^2(\Omega),\qquad\mu\in\cmr,
\end{equation*}
from which
we then conclude that $R_0$ is a selfadjoint operator in $L^2(\Omega)$. 
Summing up we have shown that $A$ is a symmetric operator which contains the selfadjoint operator $R_0$, so that $A=R_0$ is selfadjoint.
\end{proof}

Finally we state a result on the spectral properties of the operator $A$. Our proof is a variant of \cite[Lemma 2.3]{AGMST10}, see also \cite{K47}.

\begin{proposition}\label{mainprop}
 Let $A$ be the selfadjoint operator from Theorem~\ref{mainthm}. Then $0$ is an isolated eigenvalue of infinite multiplicity and the corresponding
 eigenspace is given by $\ker R^*$. 
The spectrum in $\dR\backslash\{0\}$ is discrete 
(i.e.\ composed of isolated eigenvalues of finite multiplicities)
and accumulates to $+\infty$ and $-\infty$. 
\end{proposition}

\begin{proof}
It is clear that the eigenspace $\ker A=\ker R^*$ is an infinite dimensional closed subspace of $L^2(\Omega)$. 
Moreover, 
\begin{equation}
 \cH:=\ran A=(\ker A)^\bot=(\ker R^*)^\bot =\ran R
\end{equation}
is closed according to Lemma~\ref{rrlem}~(ii).
In the following we denote the orthogonal
projection onto the subspace $\cH$ by $P$ and the embedding of $\cH$ into $L^2(\Omega)$ is denoted by $\iota$.  
For the restriction of $A$ to $\cH$ we write $A^\prime$. Note that $A^\prime$ is a bijective selfadjoint operator in 
the Hilbert space $\cH$, so that $0\not\in\sigma(A^\prime)$. With respect to the decomposition $L^2(\Omega)=\cH\oplus\cH^\bot$ we have $A=A^\prime\oplus 0$
and hence 
\begin{equation}\label{af}
 Af=\iota A^\prime Pf,\qquad f\in\dom A.
\end{equation}
It will also be used below that the orthogonal sum $A_D=A_{D+}\oplus A_{D-}$ 
of the Dirichlet operators $A_{D\pm}$ is a selfadjoint operator in $L^2(\Omega)$ and that $0\not\in\sigma(A_D)$.

Let now $f=f_R+f_0\in\dom A$, where $f_R\in\dom R$ and $f_0\in\ker A$. As $R\subset A_D$ and $R\subset A$ we have 
\begin{equation*}
f=f_R+f_0=A_D^{-1} Rf_R+f_0 =A_D^{-1} A f_R+f_0 = A_D^{-1} A  f +f_0 
\end{equation*}
and hence
\begin{equation*}
 Pf=P\bigl(A_D^{-1} A f +f_0 \bigr)=PA_D^{-1} A f=PA_D^{-1} \iota A^\prime P f,
\end{equation*}
where we have used \eqref{af} in the last equality.
This leads to 
$$A^{\prime -1}( A^\prime P f)=Pf= PA_D^{-1} \iota (A^\prime P f)$$
and as $0\not\in\sigma(A^\prime)$ we conclude 
\begin{equation*}
 A^{\prime -1}=PA_D^{-1} \iota.
\end{equation*}
Since $A_D^{-1}$ is a compact operator in $L^2(\Omega)$ it follows that $A^{\prime -1}$ is a compact operator in $\cH$. Moreover,
for $g\in\cH$ we have
\begin{equation}\label{aaaform}
 (A^{\prime -1} g,g)_\cH=(PA_D^{-1} \iota g,g)_\cH= (A_D^{-1}\iota g,\iota g)_{L^2(\Omega)}.
\end{equation}
Since $S_+\oplus S_-\subset R$ we conclude for all $f_\pm\in\dom S_\pm=H^2_0(\Omega_\pm)$
$$(S_+ f_+, 0)^\top\in\ran R=\cH\quad\text{and}\quad (0,S_- f_-)^\top\in\ran R=\cH.$$ 
It follows that
the spaces $\cH\cap L^2(\Omega_\pm)$ are both infinite dimensional. It is clear that the form on the right hand side of \eqref{aaaform}
is positive (negative) for functions in $\cH\cap L^2(\Omega_+)$ (resp.\ $\cH\cap L^2(\Omega_-))$. This implies that the positive and negative
spectra of $A^{\prime -1}$ are both infinite. Now it follows from the compactness that the spectrum of $A^\prime$ (and hence of $A$) 
in $\dR\backslash\{0\}$ is discrete and accumulates to $+\infty$ and $-\infty$.
\end{proof}

\section{Quantitive spectral properties of the selfadjoint operator $A$}\label{Sec.spec}

According to Proposition~\ref{mainprop} the spectrum of the selfadjoint operator $A$ consists of eigenvalues which accumulate to
$+\infty$ and $-\infty$.
The eigenvalue $0$ is of infinite multiplicity, the multiplicities of the nonzero eigenvalues are finite. In the next proposition we identify
the eigenvalues of $A$ with the roots of an elementary algebraic equation and we specify the eigenfunctions of $A$. 
\begin{proposition}\label{Prop.spec}
Let $A$ be the selfadjoint operator from Theorem~\ref{mainthm}. 
Then the following hold.
\begin{itemize}
\item[\textrm{(i)}]
The spectrum of~$A$ is symmetric with respect to~$0$.
\item[\textrm{(ii)}]
We have
$$
  \sigma(A) = \bigcup_{n=1}^\infty \bigcup_{m=-\infty}^\infty
  \{\lambda_{n,m}\}
  \,,
$$
where $\{\lambda_{n,m}\}_{m\in\mathbb{Z}}$ for each fixed~$n \in \mathbb{N}$
is an increasing sequence of simple roots of the algebraic equation
\begin{equation}\label{algebraic}
  \frac{\tanh\sqrt{\lambda+(n\pi)^2}}{\sqrt{\lambda+(n\pi)^2}}
  = \frac{\tan\sqrt{\lambda-(n\pi)^2}}{\sqrt{\lambda-(n\pi)^2}}
\end{equation}
for $\lambda \not= \pm (n\pi)^2$.
We arrange the sequence in such a way that $\lambda_{n,0}=0$
(zero is a solution of~\eqref{algebraic} for any $n \in \mathbb{N}$).
\item[\textrm{(iii)}]
Given any $n \in \mathbb{N}$, \eqref{algebraic} has no root in 
$(-(n\pi)^2,0)\cup(0,(n\pi)^2)$. 
In particular, $[-\pi^2,0) \cup (0,\pi^2] \not \in \sigma(A)$.
\item[\textrm{(iv)}]
The eigenfunction of~$A$ corresponding to $\lambda_{n,m}$ 
is given by $f_{n,m}(x,y) = \psi_{n,m}(x) \chi_n(y)$, 
where $\chi_n(y) = \sqrt{2} \sin(n \pi y)$ and
\begin{equation}\label{efs.1D}
  \psi_{n,m}(x) = 
  \begin{cases}
    N_{n,m} \, \sinh\sqrt{\lambda_{n,m}+(n\pi)^2} \,
    \sin\left(\sqrt{\lambda_{n,m}-(n\pi)^2} \, (1-x)\right) ,
    & x>0 \,,
    \\
    N_{n,m} \, \sin\sqrt{\lambda_{n,m}-(n\pi)^2} \,
    \sinh\left(\sqrt{\lambda_{n,m}+(n\pi)^2} \, (1+x)\right) ,
    & x<0 \,,
  \end{cases}
\end{equation}
with any $N_{n,m} \in \mathbb{C}\setminus\{0\}$.
With the normalization constants~$N_{n,m}$ satisfying
\begin{multline*}
  |N_{n,m}|^{-2} = 
  \mbox{$\sinh^2\sqrt{\lambda_{n,m}+(n\pi)^2}$}
  \left[\frac{1}{2}-\frac{\sin\left(2\sqrt{\lambda_{n,m}-(n\pi)^2}\right)}
  {4\sqrt{\lambda_{n,m}-(n\pi)^2}}\right]
  \\
  + \mbox{$\sin^2\sqrt{\lambda_{n,m}-(n\pi)^2}$}
  \left[-\frac{1}2{}+\frac{\sinh\left(2\sqrt{\lambda_{n,m}+(n\pi)^2}\right)}
  {4\sqrt{\lambda_{n,m}+(n\pi)^2}}\right]
  ,
\end{multline*}
the functions $f_{n,m}$ ($n \in \mathbb{N}$, $m \in \mathbb{Z}$) 
form a complete orthonormal set in $L^2(\Omega)$. 
\end{itemize}
\end{proposition}
\begin{proof}
The eigenvalues~$\lambda$ and the corresponding eigenfunctions~$f$ of~$A$  
can be obtained as nontrivial solutions of the differential equations
$\mp\Delta f_\pm = \lambda f_\pm$ in $\Omega_\pm$,
subject to the boundary and interface conditions determined in~\eqref{opa}.
From this boundary transmission problem, 
it is immediately seen that if~$\lambda$ is an eigenvalue of~$A$
(with eigenfunction~$f(x,y)$), 
then also~$-\lambda$ is an eigenvalue of~$A$
(with eigenfunction~$f(-x,y)$).
This establishes~(i).

The other properties (ii)--(iv) 
are obtained by a separation of variables.
Decomposing any eigenfunction $f \in L^2(\Omega)$ of~$A$
into the transverse orthonormal Dirichlet basis $\{\chi_n\}_{n=1}^\infty$,
i.e.,
$$
  f(x,y) = \sum_{n=1}^\infty \psi_n(x) \chi_n(y) \,,
  \qquad \chi_n(y) = \sqrt{2} \sin(n \pi y)\,,
$$
we easily obtain from the boundary transmission problem in~$\Omega$ 
that the function
$
  \psi_n = (\psi_{n+},\psi_{n-})^\top 
  \in L^2((0,1)) \times L^2((-1,0))
$ 
for each fixed $n \in \mathbb{N}$
is a nontrivial solution of the following problem
\begin{equation}\label{bv}
\begin{split}
  -\psi_{n+}^{\prime\prime} & = (\lambda-(n\pi)^2) \psi_{n+}\quad \mbox{in}\quad (0,1) \,,\\
   \psi_{n-}'' &= (\lambda+(n\pi)^2) \psi_{n-} \quad \mbox{in}\quad (-1,0) \,,
\end{split}
\end{equation}
subject to the boundary and interface conditions
\begin{equation}\label{bv2}
 \psi_{n+}(1)=\psi_{n-}(-1)=0,\quad 
\psi_{n+}(0)=\psi_{n-}(0),\quad\text{and}\quad \psi_{n+}^\prime(0)
=-\psi_{n-}^\prime(0) .
\end{equation}
Solving the differential equations in~\eqref{bv} in terms of exponentials 
and subjecting the latter to the boundary and interface conditions \eqref{bv2},
we find that any nontrivial solution~$\psi_n$ is of the form~\eqref{efs.1D} 
with the constrain that the eigenvalue~$\lambda$ solves~\eqref{algebraic}.
There is an infinite number of such solutions 
because~\eqref{algebraic} always contains an oscillatory tangent function
for large values of~$\lambda$. 
For each fixed $n \in \mathbb{N}$,
we arrange the roots of~\eqref{algebraic} in an increasing sequence
$\{\lambda_{n,m}\}_{m\in\mathbb{Z}}$ such that $\lambda_{n,0}=0$. 
Notice that $\lambda = \pm (n\pi)^2$ are not admissible solutions
of~\eqref{bv} for any $n \in \mathbb{N}$.
This is in fact consistent with~\eqref{algebraic}, 
because the limit $\lambda \to \pm (n\pi)^2$ 
casts~\eqref{algebraic} into
$
  \tanh\sqrt{2(n\pi^2)} = \sqrt{2(n\pi^2)}
$
which is never satisfied for nonzero~$n$.
We have thus proved~(ii), except for the simplicity 
of the roots of~\eqref{algebraic},
which will be established at the end of this proof.
As for~(iv), it only remains to recall that eigenfunctions 
of a selfadjoint operator with pure point spectrum 
form a complete orthonormal set when normalized properly
($N_{n,m}$ is chosen in such a way that all~$\psi_{n,m}$
have norm~$1$ in $L^2((-1,1))$ and $\chi_n$ are already
normalized to~$1$ in $L^2((0,1))$).  

Now we turn to a proof of~(iii). 
Recall that we already know 
that no eigenvalue can be equal to $\pm (n\pi)^2$, with $n \in \mathbb{N}$. 
To show that~\eqref{algebraic} has no root in $(0,(n\pi)^2)$,
it is enough to show that the function 
$$
  G(\lambda) = 
  \frac{\sqrt{\lambda+(n\pi)^2}}{\tanh\sqrt{\lambda+(n\pi)^2}}
  - \frac{\sqrt{(n\pi)^2-\lambda}}{\tanh\sqrt{(n\pi)^2-\lambda}}
$$
does not vanish in $(0,(n\pi)^2)$.
This follows from $G(0)=0$ and 
\begin{multline*}
  G'(\lambda) = \frac{1}{4} \Bigg[
  \frac{\sinh\left(2\sqrt{\lambda+(n\pi)^2}\right)-2\sqrt{\lambda+(n\pi)^2}}
  {\sqrt{\lambda+(n\pi)^2} \, \sinh^2\sqrt{\lambda+(n\pi)^2}}
  \\
  + \frac{\sinh\left(2\sqrt{(n\pi)^2-\lambda}\right)-2\sqrt{(n\pi)^2-\lambda}}
  {\sqrt{(n\pi)^2-\lambda} \, \sinh^2\sqrt{(n\pi)^2-\lambda}}
  \Bigg] > 0
  ,
\end{multline*}
for $\lambda \in (0,(n\pi)^2)$, where the crucial inequality 
is due to
the elementary bound $\sinh(x) > x$ valid for all $x>0$.
Since~\eqref{algebraic} is symmetric 
with respect to the change $\lambda \mapsto -\lambda$, 
the claim on the absence of roots extends
to the symmetric set $(-(n\pi)^2,0)\cup(0,(n\pi)^2)$.

It remains to prove the simplicity of roots stated in~(ii).
By symmetry of~\eqref{algebraic}, it is again enough to show it
for non-negative roots~$\lambda_{n,m}$ only.
Defining 
\begin{equation}\label{F}
  F(\lambda) 
  = \frac{\tanh\sqrt{\lambda+(n\pi)^2}}{\sqrt{\lambda+(n\pi)^2}}
  - \frac{\tan\sqrt{\lambda-(n\pi)^2}}{\sqrt{\lambda-(n\pi)^2}}
  \,,
\end{equation}
we have that
$\lambda_{n,m}$ is a root of~\eqref{algebraic} if, and only if,
$F(\lambda_{n,m})=0$.
Using this identity, 
it is straightforward to cast the derivative of~$F$ at $\lambda_{n,m}$
into the form
$$
%\begin{aligned}
  F'(\lambda_{n,m}) = %&
  - \frac{\tanh^2\sqrt{\lambda_{n,m}+(n\pi)^2}}{\lambda_{n,m}+(n\pi)^2}
  \\
  %& 
  + \frac{(n\pi)^2}{\lambda_{n,m}^2-(n\pi)^4}
  \left[
  \frac{\tanh\sqrt{\lambda_{n,m}+(n\pi)^2}}{\sqrt{\lambda_{n,m}+(n\pi)^2}}
  - 1
  \right]
  .
%\end{aligned}
$$
If $\lambda_{n,m} > 0$, then we know by~(iii) 
that necessarily $\lambda_{n,m} > (n\pi)^2$.
Using the elementary bound $\tanh(x) < x$ for all $x>0$,
we thus obtain 
$$
  F'(\lambda_{n,m}) 
  < - \frac{\tanh^2\sqrt{\lambda_{n,m}+(n\pi)^2}}{\lambda_{n,m}+(n\pi)^2}
  < 0
  \,.
$$
On the other hand, 
employing standard algebraic expressions for hyperbolic functions, 
it is easy to check that the formula for $F'(\lambda_{n,m})$ above
reduces for $\lambda_{n,0}=0$ to
$$
  F'(0) = \frac{2n\pi-\sinh(2n\pi)}{2 (n\pi)^3 \cosh^2(n\pi)} < 0
  \,,
$$
where the inequality follows by the elementary bound used above
in the proof of~(iii).
Summing up, $F'(\lambda)\not=0$ whenever $F(\lambda)=0$,
which proves the simplicity of the roots of~\eqref{algebraic}
and completes the proof of the proposition.
\end{proof}

We remark that the simplicity of roots of~\eqref{algebraic}
stated in point~(ii) of the above proposition
does not mean that the eigenvalues of~$A$ simple.
In fact, we already know from Proposition~\ref{mainprop}
that~$0$ is an eigenvalue of infinite multiplicity.

In order to establish the convergence results 
announced in the introduction in a unified way,
we consider now a more general situation of the differential expression
\begin{equation}\label{A.complex.bis}
 \cT_\delta f=-\text{div}\, (a_\delta\,\nabla f),
  \qquad \quad 
  a_\delta(x,y)=
  \begin{cases} 1, & 
  (x,y)\in\Omega_+, \\ 
  \displaystyle
  -\frac{1}{1+\delta}, & (x,y)\in\Omega_-, 
 \end{cases}
\end{equation}
where~$\delta$ is an arbitrary complex number with $|\delta| < 1$.
We also introduce an associated operator 
\begin{equation}\label{A.complex.bis.op}
  T_\delta f = \cT_\delta f, 
  \qquad
  \dom T_\delta = \bigl\{f \in H_0^1(\Omega) : \cT_\delta f \in L^2(\Omega)\bigr\}.
\end{equation}
Clearly, by choosing~$\delta$ appropriately,
the eigenvalue problems 
for the selfadjoint operator~$A_\varepsilon$ from~\eqref{A_eps.op} 
and the (up to a rotation) $m$-sectorial operator~$B_\eta$ from~\eqref{A.complex.op}
can be cast into the form of the eigenvalue problem for~$T_\delta$. 
The latter reads
\begin{equation}\label{bv.general}
\begin{split}
  -\Delta f_+ &= \lambda f_+\quad\qquad\quad \mbox{in}\quad \Omega_+ \,,\\
  \Delta f_- &= (1+\delta)\lambda f_- \quad\mbox{in}\quad \Omega_- \,,
\end{split}
\end{equation}
 where, in addition, 
$f=(f_+,f_-)^{\txtD{\top}}\in \dom T_\delta\subset H_0^1(\Omega)$ 
satisfies the interface  condition 
\begin{equation}
  (1+\delta)\partial_{{\bf n}_+} f_+\vert_\cC 
  = \partial_{{\bf n}_-} f_-\vert_\cC \,.  
\end{equation}
\begin{proposition}\label{Prop.spec.general}
Let $T_\delta$ be the operator introduced in~\eqref{A.complex.bis.op}. 
There exists an absolute constant $c>0$ 
such that for $|\delta| \leq c$ the following hold.
\begin{itemize}
\item[\textrm{(i)}]
We have
$$
  \sigma_\mathrm{p}(T_\delta) = 
  \bigcup_{n=1}^\infty \bigcup_{m=-\infty}^\infty
  \{\lambda_{n,m}^\delta\}
  \,,
$$
where $\{\lambda_{n,m}^\delta\}_{m\in\mathbb{Z}}$ 
for each fixed~$n \in \mathbb{N}$
is a sequence of roots of the algebraic equation
\begin{equation}\label{algebraic.general}
  (1+\delta) \,
  \frac{\tanh\sqrt{(1+\delta)\lambda+(n\pi)^2}}
  {\sqrt{(1+\delta)\lambda+(n\pi)^2}}
  = \frac{\tan\sqrt{\lambda-(n\pi)^2}}{\sqrt{\lambda-(n\pi)^2}}
\end{equation}
for $\lambda \not= (n\pi)^2$ and $\lambda \not= -(n\pi)^2/(1+\delta)$.
\item[\textrm{(ii)}]
The eigenfunction of~$T_\delta$ corresponding to~$\lambda_{n,m}^\delta$ 
is given by $f_{n,m}^\delta(x,y) = \psi_{n,m}^\delta(x) \chi_n(y)$, 
where $\chi_n(y) = \sqrt{2} \sin(n \pi y)$ and
\begin{equation}\label{efs.1D.general}
  \psi_{n,m}^\delta(x) = 
  \begin{cases}
    N_{n,m}^\delta \, \sinh\sqrt{(1+\delta)\lambda_{n,m}^\delta+(n\pi)^2} \,
    \sin\left(\sqrt{\lambda_{n,m}^\delta-(n\pi)^2} \, (1-x)\right) ,
    & x>0 ,
    \\
    N_{n,m}^\delta \, \sin\sqrt{\lambda_{n,m}^\delta-(n\pi)^2} \,
    \sinh\left(\sqrt{(1+\delta)\lambda_{n,m}^\delta+(n\pi)^2} \, (1+x)\right) ,
    & x<0 ,
  \end{cases}
\end{equation}
with any $N_{n,m}^\delta \in \mathbb{C}\setminus\{0\}$.
With the normalization constants~$N_{n,m}^\delta$ satisfying

\begin{eqnarray*}
  \lefteqn{|N_{n,m}^\delta|^{-2} = 
  \mbox{$\left|\sinh\left(\sqrt{(1+\delta)\lambda_{n,m}^\delta+(n\pi)^2}\right)\right|^2$}
  }
  \\
  && \quad \times
  \left[\frac{\sinh\left(2\IM\sqrt{\lambda_{n,m}^\delta-(n\pi)^2}\right)}
  {4\IM\sqrt{\lambda_{n,m}^\delta-(n\pi)^2}}
  -\frac{\sin\left(2\RE\sqrt{\lambda_{n,m}^\delta-(n\pi)^2}\right)}
  {4\RE\sqrt{\lambda_{n,m}^\delta-(n\pi)^2}}\right]
  \\
  && + \mbox{$\left|\sin\left(\sqrt{\lambda_{n,m}^\delta-(n\pi)^2}\right)\right|^2$}
  \\
  && \quad \times
  \left[-\frac{\sin\left(2\IM\sqrt{(1+\delta)\lambda_{n,m}^\delta+(n\pi)^2}\right)}
  {4\IM\sqrt{(1+\delta)\lambda_{n,m}^\delta+(n\pi)^2}}
  +\frac{\sinh\left(2\RE\sqrt{(1+\delta)\lambda_{n,m}^\delta+(n\pi)^2}\right)}
  {4\RE\sqrt{(1+\delta)\lambda_{n,m}^\delta+(n\pi)^2}}\right]
  ,
\end{eqnarray*}

the functions $f_{n,m}^\delta$ ($n \in \mathbb{N}$, $m \in \mathbb{Z}$) 
are normalized to~$1$ in $L^2(\Omega)$. 
\end{itemize}
\end{proposition}
\begin{proof}
The results follow by the separation of variables 
as in the proof of Proposition~\ref{Prop.spec}.
Contrary to the symmetric situation $\delta=0$, however,
\eqref{algebraic.general} can have solutions 
$\lambda = (n\pi)^2$ and $\lambda = -(n\pi)^2/(1+\delta)$.
Compatibility conditions for the existence of such solutions are
\begin{equation}\label{compatibility}
%\begin{aligned}
  \frac{\tanh\sqrt{(2+\delta)(n\pi)^2}}{\sqrt{(2+\delta)(n\pi)^2}}
  = \frac{1}{1+\delta} 
  \,, \qquad
  \frac{\displaystyle \tanh\sqrt{\frac{2+\delta}{1+\delta}\,(n\pi)^2}}
  {\displaystyle \sqrt{\frac{2+\delta}{1+\delta}\,(n\pi)^2}}
  = 1+\delta 
  \,,
%\end{aligned}
\end{equation}
respectively (they can be obtained from~\eqref{algebraic.general}
after the limit $\lambda \to (n\pi)^2$ and $\lambda \to -(n\pi)^2/(1+\delta)$,
respectively). 
We claim that these ``exceptional'' solutions do not exist 
for all~$\delta$ small in the absolute value,
uniformly in $n \in \mathbb{N}$. 
This can be proved straightforwardly by comparing the real parts of 
the left and right sides of~\eqref{compatibility}.
More specifically, we have 
$$
  \left|\RE \left( \frac{\tanh z}{z} \right)\right|
  = \frac{1}{|z|^2} 
  \left|
  \frac{z_1 \sinh(2 z_1) + z_2 \sin(2 z_2)}
  {\cosh(2 z_1)+\cos(2 z_2)}
  \right|
  \leq \frac{1}{|z_1|} 
  \frac{\sinh(2 |z_1|) + 1}
  {\cosh(2 |z_1|)-1}
$$
for all $z = z_1 + i z_2 \in \mathbb{C}$ with $z_1=\RE z \not=0$,
where the right hand side is decreasing as a function of~$|z_1|$.
Employing the elementary inequality
$
  |\RE\sqrt{\xi}| \geq |\sqrt{\RE\xi}|
$
valid for every $\xi \in \mathbb{C}$ with $\RE\xi \geq 0$
and $|\delta| < 1$,
we estimate
\begin{align*}
  \left|\RE\sqrt{(2+\delta)(n\pi)^2}\right| 
  &\geq \left|\sqrt{(2+\RE\delta)(n\pi)^2}\right| 
  \geq \pi \,,
  \\
  \left|\RE\sqrt{\frac{2+\delta}{1+\delta}\,(n\pi)^2}\right|
  &\geq \left|\sqrt{\left(1+\frac{1+\RE\delta}{1+2\RE\delta+|\delta|^2}\right) 
  (n\pi)^2}\right|
  \geq \pi \,.
\end{align*}
Consequently, we see that a necessary condition for
an equality to hold in~\eqref{compatibility} is
$$
  0.32 \approx
  \frac{1}{\pi} \frac{\sinh(2\pi)+1}{\cosh(2\pi)-1}
  \geq \min\left\{
  \RE\left(\frac{1}{1+\delta}\right), \RE(1+\delta)
  \right\}
  \geq \frac{1-|\delta|}{(1+|\delta|)^2}
  \,,
$$
which is clearly impossible if~$|\delta|$ is small enough
(the present estimates yield $c \geq 0.38$).
\end{proof}

Now we are in a position to establish the convergence of
eigenvalues and eigenfunctions of~$T_\delta$ 
to eigenvalues and eigenfunctions of~$A$ as $\delta \to 0$. 
In the next theorem we show,  in particular, that the operators ~$A_\varepsilon$ and~$B_\eta$ in the introduction  
represent an ``approximation'' of the selfadjoint operator~$A$, 
at least on the spectral level.
However, the resolvents of $A_\varepsilon$ and $B_\eta$ 
are compact for all $\kappa\not=1$ and $\eta>0$,
while the resolvent of $A$ is not compact (zero is an eigenvalue of infinite multiplicity).

\begin{theorem}
For $n \in \mathbb{N}$ and $m \in \mathbb{Z}$,
let~$\lambda_{n,m}$ and~$\psi_{n,m}$ be respectively the eigenvalues
and eigenfunctions of~$A$ specified in Proposition~\ref{Prop.spec}
and let~$\lambda_{n,m}^\delta$ and~$\psi_{n,m}^\delta$ 
be respectively the eigenvalues
and eigenfunctions of~$T_\delta$ specified
in Proposition~\ref{Prop.spec.general}.
For any $n \in \mathbb{N}$, 
the sequence $\{\lambda_{n,m}^\delta\}_{m \in \mathbb{Z}}$
can be arranged in such a way that 
$$
  \lim_{\delta \to 0} 
  \big|\lambda_{n,m}^\delta - \lambda_{n,m}\big| = 0
  \qquad \mbox{and} \qquad
  \lim_{\delta \to 0} 
  \big\|\psi_{n,m}^\delta - \psi_{n,m}\big\|_{L^\infty(\Omega)} = 0
  \,.
$$
\end{theorem}
\begin{proof}
The convergence of eigenvalues follows by the implicit function theorem 
applied to 
$$
  H(\lambda,\delta) = 
  (1+\delta) \,
  \frac{\tanh\sqrt{(1+\delta)\lambda+(n\pi)^2}}
  {\sqrt{(1+\delta)\lambda+(n\pi)^2}}
  - \frac{\tan\sqrt{\lambda-(n\pi)^2}}{\sqrt{\lambda-(n\pi)^2}}
  \,.
$$
Clearly, $H(\lambda,0)=F(\lambda)$, where~$F$ is introduced in~\eqref{F}
based on~\eqref{algebraic}. Hence, $H(\lambda_{n,m},0)=0$. 
We only need to check that the derivative $\partial_1 H(\lambda_{n,m},0)$
does not vanish.
However,  $\partial_1 H(\lambda_{n,m},0) = F'(\lambda_{n,m}) \not= 0$,
due to the proof of simplicity of the roots of~\eqref{algebraic} 
established in the proof of Proposition~\ref{Prop.spec}.
The convergence of eigenfunctions is then clear from
the expressions~\eqref{efs.1D} and~\eqref{efs.1D.general}.  
\end{proof}

\end{document}